\documentclass[10pt]{article}
\usepackage{amssymb}
\usepackage{amsmath}
\usepackage{amsthm}
\usepackage[textwidth=16cm, textheight=24cm]{geometry}
\usepackage[numbers]{natbib}

\usepackage[english]{babel}
\usepackage[utf8]{inputenc}
\usepackage[T1]{fontenc}
\usepackage{todonotes}
\usepackage{xcolor}
\usepackage[all]{xy}
\usepackage{booktabs}
\usepackage{tikz-cd}

\usepackage[pdfborder={0 0 0}]{hyperref}
\usepackage{setspace}
\newtheorem{thm}{Theorem}[section]
\newtheorem{claim}{CLAIM}
\newtheorem{cor}[thm]{Corollary}
\newtheorem{lem}[thm]{Lemma}

\newtheorem{prop}[thm]{Proposition}
\theoremstyle{definition}
\newtheorem{defn}[thm]{Definition}
\newtheorem{remark}[thm]{Remark}
\newtheorem{example}[thm]{Example}
\newtheorem{construction}[thm]{Construction}
\newtheorem{convention}[thm]{Convention}

\def\dim{\operatorname{dim}}

\def\det{\operatorname{det}}

\def\Spec{\operatorname{Spec}}

\def\tensor{\otimes}
\def\onto{\twoheadrightarrow}
\def\into{\hookrightarrow}

\def\Sym{\operatorname{Sym}}%

\def\VV{V}%
\def\cubicspace{\mathbb{P}\left(\Sym^3 \VV\right)}%
\def\Hilbshort{\mathcal{H}}%
\def\Hilbshortzero{\Hilbshort_{gen}}%
\def\Hilbshortother{\Hilbshort_{1661}}%
\def\exceptional{\Hilbshort^{gr}_{1661}}%

\def\OO{\mathcal{O}}%
\def\DIR{D_{IR}}%
\def\DVap{D_{V-ap}}%
\def\Sing{\operatorname{Sing}}%
\def\gr{\operatorname{gr}}%
\DeclareMathOperator{\hook}{\lrcorner}
\def\spann#1{\operatorname{span}\left( #1 \right)}%
\newcommand{\SL}[1]{\operatorname{SL}(#1)}%
\DeclareMathOperator{\VSP}{VSP}%
\newcommand{\Ann}[1]{\operatorname{Ann}\left( #1 \right)}%
\newcommand{\Apolar}[1]{\operatorname{Apolar}\left( #1 \right)}%
\newcommand{\Dx}{\alpha}

\begin{document}

\title{VSPs of cubic fourfolds and the Gorenstein locus of the Hilbert scheme of $14$ points on
    $\mathbb{A}^6$}
    \author{Joachim Jelisiejew\thanks{While preparing this paper, the author
        was a doctoral fellow at the Warsaw Center of
Mathematics and Computer Science, University of Warsaw and was financed by the Polish program KNOW
and grant 2014/13/N/ST1/02640.\newline
Current email: jjelisiejew@impan.pl}}
    \maketitle

    \begin{abstract}
        To a cubic fourfold one may associate a geometric object (a
        hyperk\"ahler manifold) via the
        theory of VSP and an algebraic object (a finite Gorenstein
        algebra) via apolarity.
        We prove that the associated algebra is smoothable
        if and only if the fourfold lies on the Iliev-Ranestad divisor (which parameterizes
        certain cubics whose $\VSP$ is isomorphic to the Hilbert
        scheme of two points on a K3 surface).
        This bridge allows us to give a detailed description of the algebraic side, i.e., the
        Gorenstein locus of 14 points on $\mathbb{A}^6$ and also to identify the
        equation of the Iliev-Ranestad divisor as the unique degree $10$ invariant of
        $\operatorname{SL}_6$.
        As our main technical tool, we develop a relative version of apolarity,
        building on ideas of Elias-Rossi.
    \end{abstract}

    \section{Introduction}

    Let $V$ be a six-dimensional $\mathbb{C}$-vector space and let $F\in
    \Sym^3 V$ be a cubic.

    Let $(F = 0) \subset \mathbb{P}V^{*}  \simeq \mathbb{P}^5$ be a cubic fourfold. It is classically
    known that for general $F$ the Fano variety of lines on $(F= 0)$ is a hyperk\"ahler fourfold.
    In~\cite{Ranestad_Iliev__VSPcubic, Ranestad_Iliev__VSPcubic_addendum,
    Ranestad_Voisin__VSP} another hyperk\"ahler fourfold is constructed and
    investigated. Namely, one considers all
    $10$-tuples $\lambda_1, \ldots ,\lambda_{10}$ of points of $\mathbb{P}V$ such that
    $F\in \spann{\lambda_1^3, \ldots ,\lambda_{10}^3}$. Such tuple $\Gamma =
    \{\lambda_1, \ldots , \lambda_{10}\}$ gives a point of the Hilbert scheme of
    $10$ points in $\mathbb{P}V$ and the closure of
    the set of $\Gamma$'s is denoted by $\VSP(F, 10)$. For general $F$ the
    variety $\VSP(F, 10)$ this is a smooth hyperk\"ahler
    fourfold~\cite{Ranestad_Iliev__VSPcubic}. The name VSP stands for Variety
    of Sums of Powers. To prove that $\VSP(F, 10)$ is hyperk\"ahler, Ranestad
    and Iliev construct a divisor $\DIR$ in the space of cubic fouldfolds,
    such that for $[F]\in \DIR$ the variety $\VSP(F, 10)$ is isomorphic to a
    Hilbert scheme of two points on a K3 surface,
    see~\cite[Theorem~3.7, Proposition~3.11]{Ranestad_Iliev__VSPcubic}.

    In a more algebraic flavour, we have the apolarity (or contraction) action
    $(-)\hook(-)\colon\Sym V^* \tensor \Sym V \to \Sym V$. This
    action is easier to describe after a choice of basis $x_1, \ldots ,x_6$ of
    $V$. Then we also have a dual basis $\Dx_1, \ldots ,\Dx_n$ of $V^*$ and
    $\Sym V = \mathbb{C}[x_1, \ldots ,x_n]$, $\Sym V^* = \mathbb{C}[\Dx_1, \ldots ,\Dx_n]$.
    We define $(-)\hook(-)$ on monomials by the rule
    \begin{equation}\label{eq:apolaritydef}
        (\Dx_1^{a_1} \ldots \Dx_n^{a_n})\hook (x_1^{b_1} \ldots x_n^{b_n}) =
        \begin{cases}
            x_1^{b_1-a_n} \ldots x_n^{b_n-a_n} & \mbox{ if all } b_i\geq a_i\\
            0 & \mbox{otherwise}.
        \end{cases}
    \end{equation}
    and extend it bi-linearly to all polynomials.
    For $G\in \Sym V$, we define $\Ann{G} := \left\{ P\ |\
    P\hook G = 0\right\}$ and $\Apolar{G} := \Sym V^* / \Ann{G}$. This is a
    finite local Gorenstein algebra, corresponding to a scheme $\Spec
    \Apolar{G}\subset V$ supported at the origin.
    In particular for a cubic $F$ which is not a cone the algebra
    $\Spec\Apolar{F}$ has degree $14$, hence gives a point on the Hilbert
    scheme of $14$ points on $\mathbb{A}^6$. Since $\Apolar{F}$ is
    Gorenstein, the point $\Spec \Apolar{F}$ lies in the \emph{Gorenstein locus} of this Hilbert
    scheme.

    In this paper we investigate the geometry of the Gorenstein locus near
    $\Spec \Apolar{F}$ and in fact we obtain a fairly full description of the
    Gorenstein locus for $14$ points of $\mathbb{A}^6$. This is one of the very few
    cases, when such description is known; indeed little is known about the
    Hilbert schemes of points on higher dimensional affine or projective
    spaces, despite much research on their singularities
    \cite{emsalem_iarrobino_small_tangent_space,
        iarrobino_kanev_book_Gorenstein_algebras,
    erman_Murphys_law_for_punctual_Hilb} and components~\cite{cn09, cn10,
        dosSantos, Sivic__Varieties_of_commuting_matrices, DJNT,
    Jelisiejew__Elementary}. The interest in the Hilbert schemes of points and
    especially their Gorenstein loci is recently renewed
    by their connections with secant varieties and thus with classical projective geometry,
    see~\cite{bubu2010, LO, iarrobino_kanev_book_Gorenstein_algebras,
    Galazka_VB_eqs_of_cacti}. Note that the Hilbert scheme of $\mathbb{P}^6$
    is covered by the Hilbert schemes of $\mathbb{A}^6$ where $\mathbb{A}^6$
    are standard open subsets. Hence the difference between
    $\mathbb{A}^6$ and $\mathbb{P}^6$ is not crucial for our considerations. In this paper we
    have chosen to use $\mathbb{A}^6$.

    \bigskip

    Recall that $\VV$ is a $6$-dimension vector space with its natural affine
    space structure.
    Denote by $\Hilbshort$ the Gorenstein locus of the Hilbert scheme of
    $14$ points on $\VV$.
    Topologically, the scheme $\Hilbshort$ has two irreducible
    components: the smoothable component $\Hilbshortzero$ of dimension $84$ and another
    component $\Hilbshortother$ of dimension $76$, see~\cite{cjn13}.
    A general point of $\Hilbshortzero$ corresponds to a set of $14$ points on
    $\VV$. Therefore finite schemes $Z$ corresponding to points in $\Hilbshortzero$ are
    called \emph{smoothable}.


    \def\cubicspacereg{\mathbb{P}\left(\Sym^3 \VV\right)_{1661}}%
    By~\cite{cjn13} the points of $\Hilbshortother$ correspond
    bijectively to finite irreducible subschemes $Z \subset V$ such that the local
    algebra $H^0(Z, \OO_Z)$ has Hilbert function $(1, 6, 6, 1)$.
    Let $\exceptional \subset \Hilbshortother$ be the set corresponding to
    $Z$ supported at the origin of $\VV$ and invariant under the dilation
    action of $\mathbb{C}^*$,~i.e.,~such that $H^0(V, I_Z)$ is homogeneous.
    Then $\exceptional \subset \Hilbshortother$ is a closed subset.
    We endow $\Hilbshortother$ with a scheme structure, which we
    formally define as the scheme-theoretic closure $\Hilbshortother =
    \overline{\Hilbshort \setminus \Hilbshortzero}$. Under this definition it
    is not clear whether $\Hilbshortother$ is reduced, although we show that indeed
    it is (we do not know whether $\Hilbshortzero$ is reduced). We endow
    $\exceptional$ with a reduced scheme structure.


    \begin{thm}[description of the Gorenstein locus of Hilbert scheme of $14$
            points on $\mathbb{A}^6$]\label{ref:mainthm:thm}
        With notation as above, the following hold.
        \begin{enumerate}
            \item\label{it:smoothness} The component $\Hilbshortother$ is connected and smooth
                (hence reduced).
            \item\label{it:retraction} There is a morphism
                \[\pi:\Hilbshortother \to \exceptional\]
                which makes $\Hilbshortother$ the total space of a rank $21$ vector
                bundle over $\exceptional$.
            \item\label{it:exceptional} The set $\exceptional$ is canonically
                isomorphic to an open subset
                of $\cubicspace$ consisting precisely of cubics
                which are not cones. We denote this subset by
                $\cubicspacereg$.
            \item\label{it:intersection} The set theoretic intersection $\Hilbshortzero \cap \Hilbshortother$
                is a prime divisor inside $\Hilbshortother$ and it is equal to
                $\pi^{-1}(\DIR)$, where ${\DIR\subset \exceptional \subset
                \cubicspace}$ is the restriction of the Iliev-Ranestad divisor.
                We get the following diagram of vector bundles:
                \[
                    \xymatrix{\Hilbshortzero \cap \Hilbshortother   \ar[d] &
                        \subset &
                        \Hilbshortother \ar[d]\\
                    \DIR  & \subset&\cubicspacereg}
                \]
        \end{enumerate}
    \end{thm}
    The most surprising part of this theorem is that the intersection of
    $\Hilbshortzero\cap \Hilbshortother$ --- which parameterizes smoothable
    subschemes from $\Hilbshortother$ --- is essentially given by the divisor
    $\DIR$, i.e., by a geometric condition. This, along with the proof of
    reducedness of $\Hilbshortother$, are the most nontrivial parts of the
    proof.

    The map $\pi$ is defined at the level of points as follows. We take $[Z]\in
    \Hilbshortother$. Translating, we fix its support at $0\in \VV$. Then we
    replace $H^0(Z, \OO_Z)$ by its associated graded algebra which in this
    very special case is also Gorenstein. We take $\pi([Z])$ to be the point
    corresponding to $\Spec \gr H^0(Z, \OO_Z)$ supported at the origin of
    $\VV$.
    The identification of $\exceptional$ with $\cubicspacereg$ in
    Point~\ref{it:exceptional} is done canonically using Macaulay's
    inverse systems. Note that the complement of $\cubicspacereg$  has
    codimension greater than one, hence divisors on $\cubicspace$ and
    $\exceptional$ are identified via restriction and closure.
    We describe a set-theoretic  equation of $\DIR \subset \cubicspace$
    explicitly as an $120\times 120$ determinant, see the discussion before
    Lemma~\ref{ref:intersectiondegree:lem}. We also note
    that $\DIR \subset \cubicspace$ is the unique $\SL{V}$-invariant divisor
    of degree~$10$.
    The following result gives a rich source of points in the intersection
    $\Hilbshortzero \cap \exceptional$, by means of the ninth secant variety,
    which is the closure of cubics which are sums of nine cubes.
    \def\secantvariety{\sigma_9 \nu_3 (\mathbb{P}\VV)}%
    \begin{prop}[{\cite[Lemma~3.1]{CHILO__symmetrizing_matrix_multiplication}}]\label{ref:secant:prop}
        Set-theoretically, the intersection $\secantvariety \cap
        \cubicspacereg$ is contained in $\DIR$. In other words, the apolar scheme of any form $[F]\in
        \cubicspacereg$ of border rank at
        most nine is smoothable.
    \end{prop}
    In particular by~\cite{AlexHirsch}, the secant variety $\secantvariety$ is of codimension one
    in the space of smoothable cubics. There is a strong evidence that it is a
    complete intersection of $\DIR$ and another geometrically defined divisor
    $\DVap$, see Remark~\ref{ref:chilozuccero:rmk}.

    It also follows from Theorem~\ref{ref:mainthm:thm} that both components of
    $\Hilbshort$ are rational. It is an open question whether
    all components of Hilbert schemes of points are rational. Roggero and
    Lella prove that components satisfying strong smoothness assumptions are
    rational, see~\cite{Roggero_Lella__rationality}. The component
    $\Hilbshortother$ is also the unique known Gorenstein component with
    dimensional smaller than $\Hilbshortzero$, see
    Remark~\ref{ref:uniquesmall:rmk}.

    The major part of this paper is
    the proof of Theorem~\ref{ref:mainthm:thm}, which is abstracted into three
    Claims, see Section~\ref{ssec:proof}. The
    main underlying tool is given by relative Macaulay's inverse
    systems, presented in Section~\ref{sec:prelims}, which we now
    discuss. Recall the definition of the apolarity action
    from~\eqref{eq:apolaritydef}. It is classically known
    (see~\cite{emsalem, iarrobino_kanev_book_Gorenstein_algebras}) that this action
    parameterizes all finite Gorenstein subschemes as follows.
    \begin{center}
    \vspace{-\baselineskip}
        \begin{tabular}[h]{@{} p{0.3\textwidth} lll p{0.5\textwidth}}
            Finite Gorenstein subschemes &&&& Polynomials/power series\\ \midrule
            $Z \subset \mathbb{C}^n$ && is equal to && $\Spec\Apolar{F}$
            for some
            $F\in \mathbb{C}_{dp}[[x_1, \ldots ,x_n]]$\\
            $Z \subset \mathbb{C}^n$ supported at the origin && is
            equal to && $\Spec\Apolar{F}$ for some
            $F\in \mathbb{C}_{dp}[x_1, \ldots ,x_n]$\\
            $Z \subset \mathbb{C}^n$ supported at the origin,
            invariant under dilation && is
            equal to && $\Spec\Apolar{F}$ for some homogeneous
            ${F\in \mathbb{C}_{dp}[x_1, \ldots ,x_n]}$\\
        \end{tabular}
    \end{center}
    Inspired by~\cite{emsalem, Elias_Rossi__Gorenstein_for_higher_dims,
    Masuti_Tozzo__Structure_of_inverse_system},  we generalize
    the above description to \emph{families} over an affine Noetherian base
    $A$, as follows (see Section~\ref{ssec:relative} for details).
    \begin{center}
        \begin{tabular}[h]{@{} p{0.4\textwidth} lll p{0.43\textwidth}}
            Families of Gorenstein subschemes &&&& Polynomials/power series\\ \midrule
            Family of $Z \subset \mathbb{C}^n$ && is locally && $\Spec\Apolar{F}$
            for some
            $F\in A[[x_1, \ldots ,x_n]]$\\
            Family of $Z \subset \mathbb{C}^n$, each $Z$ supported at the
            origin && is locally && $\Spec\Apolar{F}$
            for some
            $F\in A[x_1, \ldots ,x_n]$\\
            Family of $Z \subset \mathbb{C}^n$, each $Z$ supported at the
            origin and invariant under dilation && is locally && $\Spec\Apolar{F}$
            for some homogeneous
            ${F\in A[x_1, \ldots ,x_n]}$\\
        \end{tabular}
    \end{center}
    The assumptions that the base is affine and that the
    description holds only locally can be removed
    at the expense of changing language to a more geometric one, for example
    one needs to  replace $F$ by a line bundle.

    \subsection*{Acknowledgements}

    This paper builds upon the beautiful math presented in the
    works of Iliev, Ranestad and
    Voisin~\cite{Ranestad_Iliev__VSPcubic, Ranestad_Voisin__VSP}, Cartwright, Erman,
    Velasco and Viray~\cite{CEVV}  and
    Elias and Rossi~\cite{EliasRossiShortGorenstein}. Additionally,
    we thank Giorgio Ottaviani for discussing questions related to
    $\secantvariety$ and Gianfranco Casnati and Roberto Notari for
    discussing irreducibility questions. We are grateful to the
    organizers of ``Theory and applications of syzygies, on the occasion of
    Frank-Olaf Schreyer’s 60th birthday'' (supported by the DFG and NSF), where the author had, among other
    inspiring activities, the possibility to enjoy Kristian Ranestad's lecture
    and discuss with Giorgio Ottaviani. Finally, we thank Jaros{\l{}}aw Buczy\'nski
    for introducing us to the subject and for his helpful comments on
    a preliminary version.

    The computations were made using Magma~\cite{Magma}, Macaulay2~\cite{M2}
    and LiE~\cite{LiE}.

    \section{Preliminaries on finite schemes}\label{sec:prelims}

    \def\mmR{\mathfrak{m}_R}%
    Let $Z$ be an irreducible finite scheme over $\mathbb{C}$. Then
    $R = H^0(Z, \OO_Z)$ is a local $\mathbb{C}$-algebra with maximal ideal
    $\mmR$. The associated graded algebra of $R$ is $\gr R =
    \bigoplus_{i\geq 0} \mmR^i/\mmR^{i+1}$. The \emph{Hilbert function} of $Z$
    (or $R$) is defined as $H(i) := \dim_{\mathbb{C}} \mmR^i/\mmR^{i+1}$.
    The algebra $R$ is Gorenstein if and only if the annihilator of $\mmR$ is
    a one-dimensional $\mathbb{C}$-vector space.  If $R$ is Gorenstein then
    the \emph{socle degree} of $R$ is the largest $d$ such that $H(d) > 0$.
    Let $R$ be Gorenstein of socle degree $d$. We will use the following
    properties of the Hilbert function (see~\cite{bubu2010, ia94}):
    \begin{enumerate}
        \item If $R$ is graded, then $H(d - i) = H(i)$ for all~$i$.
        \item If $H(d-i) = H(i)$ for all $i$, then $\gr R$ is also Gorenstein.
        \item If $H = (1, n, n, 1)$ then $\gr R$ is isomorphic to $R$,
            see~\cite{EliasRossiShortGorenstein, EliasRossi_Analytic_isomorphisms}.
    \end{enumerate}

    \subsection{Macaulay's inverse systems}
    \def\kk{\Bbbk}%
    \def\SS{S}%
    \def\PP{P}%
    \def\Hom#1{\operatorname{Hom}_{#1}}%
    \def\WW{V}%
    See~\cite{EliasRossi_Analytic_isomorphisms,
    iarrobino_kanev_book_Gorenstein_algebras, Jel_classifying} for
    general facts about Macaulay's inverse systems. In this
    section we work over an arbitrary algebraically closed field $\kk$ and
    with a finite dimensional $\kk$-vector space $\WW$ with $n=\dim \WW$.
    For the main part of the article, in Section~\ref{ssec:proof} we
    only use the case $\kk = \mathbb{C}$ and $6$-dimensional~$V$.

    Let $\SS := \Sym \WW^*$ be a polynomial ring. Then we may view $\PP =
    \Hom{\kk}(S, \kk)$ as an $\SS$-module via $(s_1f)(s_2) = f(s_1s_2)$. It is
    common to introduce dual bases $\Dx_1, \ldots ,\Dx_n$ on $\WW^*$ and
    $x_1, \ldots , x_n$ on $\WW$. Then $\SS$ becomes $\kk[\Dx_1, \ldots ,\Dx_n]$
    and $\PP = \kk_{dp}[[x_1, \ldots ,x_n]]$ the divided powers power series
    ring, with the action
    \begin{equation}\label{eq:contraction}
        (\Dx_1^{a_1} \ldots \Dx_n^{a_n})\hook (x_1^{b_1} \ldots x_n^{b_n}) =
        \begin{cases}
            x_1^{b_1-a_n} \ldots x_n^{b_n-a_n} & \mbox{ if all } b_i\geq a_i\\
            0 & \mbox{otherwise}.
        \end{cases}
    \end{equation}
    In fact $\Dx_i$ acts on $\PP$ as a derivation.
    By abuse of notation we will denote the subring $\kk_{dp}[x_1, \ldots ,x_n]
    \subset \PP$ by $\Sym \WW$. The abuse here is that $\Sym \WW$ is a divided
    power ring and not a polynomial ring. Eventually we work over fields
    of characteristic zero, where these notions agree in a non-obvious way, see
    \cite[Example~A.5, p.~266]{iarrobino_kanev_book_Gorenstein_algebras}.
    In this paper the ring structure of $\PP$ is used only a few times and the
    $\SS$-module structure is far more important for our purposes.

    If $I \subset \SS$ is homogeneous then for every $d$ by $I_d^{\perp}
    \subset \Sym^d \WW$ we denote the perpendicular space.

    Let $F\in \PP$. By $\Ann{F}$ we denote its annihilator in $\SS$ and by
    $\Apolar{F} = \SS/\Ann{F}$ its \emph{apolar algebra}. Note that
    $\SS$-modules $\Apolar{F}$ and $\SS\hook F$ are isomorphic.
    As a special case, we may take an apolar algebra of $F\in \Sym \WW = \kk_{dp}[x_1, \ldots ,x_n] \subset \PP$.
    The crucial result by Macaulay is the following theorem.
    \begin{thm}
        The algebra $\Apolar{F}$ is a local artinian Gorenstein algebra
        supported at $0\in \Spec \SS$. Conversely, every such algebra is equal to
        $\Apolar{F}$ for a (non-unique) $F\in \Sym \WW$.
    \end{thm}
    \begin{proof}[Proof-sketch]
        We include one implication of this well-known proof to motivate the relative
        version given in the next section. Namely for an artinian Gorenstein
        algebra $R = \SS/I$
        supported at the origin we find an $F\in \Sym \WW$ such that $R = \Apolar{F}$.

        Since $R$ is artinian, its dualising module $\omega_R$ is isomorphic
        to $\Hom{\kk}(R, \kk)$. Since $R$ is Gorenstein, this dualising module is
        principal (\cite[Chapter 21]{Eisenbud}). Take a generator $F$ of this
        module.
        The surjection $\SS\to R$ gives an inclusion $\omega_R = \Hom{\kk}(R,
        \kk)
        \subset \Hom{\kk}(\SS, \kk) = \PP$. In this way $\omega_R$ becomes an
        $\SS$-module, whose annihilator is $I$. Since $\omega_R$ is principal
        generated by $F$, we get that $\Ann{F} = I$. Therefore $R =
        \Apolar{F}$.
        \def\mm{\mathfrak{m}}%
        \def\len{\operatorname{len}}%
        We have $\mm^{\len R} \subset I$, where $\mm$
        is the ideal of $0\in \Spec \SS$. Hence $\mm^{\len R} \subset
        \Ann{F}$ which forces $F\in \Sym \WW$.
    \end{proof}
    The freedom of choice of $F$ is well understood, see for
    example~\cite{Jel_classifying}.

    If the algebra in question is naturally graded, we may take $F$ to be homogeneous. This
    is the case of our main interest. By $\SS_r$ we denote forms of degree
    $r$. Let $F$ be homogeneous of degree $d$ and $I = \Ann{F}$.
    Then $\SS_r\hook F$ is annihilated by $I_{d-r}$ and in fact $\SS_r\hook F =
    I_{d-r}^{\perp}$. The symmetry of the Hilbert function of $\Apolar{F}$
    shows that $\dim I_{d-r}^{\perp} = \dim I_r^{\perp}$ for all $r$.

    Numerous examples of polynomials and corresponding artinian Gorenstein
	algebras can be found in the literature, for instance in~\cite{nisiabu_jabu_kleppe_teitler_direct_sums,
	iarrobino_kanev_book_Gorenstein_algebras, cn10}.

    \begin{remark}\label{ref:associatedgraded:rem}
        If $F\in \Sym \WW$ is arbitrary, but such that $\gr \Apolar{F}$ is also
        Gorenstein, then $\gr\Apolar{F}$ is isomorphic to $\Apolar{F_d}$, where $F_d$ is the
        top degree form. In particular, if $F$ is of degree three and
        $H_{\Apolar{F}} = (1, n, n,1)$ then $\gr \Apolar{F} = \Apolar{F_3}$.
        Also it follows that $H_{\Apolar{F}} = (1, n, n,
        1)$ if and only if $H_{\Apolar{F_3}} = (1, n, n, 1)$,
        see~\cite[Chapter~1]{ia94}.
    \end{remark}

    \begin{remark}\label{ref:uniquenessofcubics:rem}
        Let $F\in \Sym \WW$ be a polynomial of degree three such that $H_{\Apolar{F}} = (1, n,
        n, 1)$ where $n = \dim \WW$; in other words the Hilbert function is maximal.
        It will be important in the following to note that $\Ann{F}$ and $\Apolar{F}$ only partially depend
        on the lower homogeneous components of $F$. To see this, it is enough
        to write explicitly the $\SS$-module $\SS\hook F$.
        \begin{align*}
            \SS\hook F &= \spann{F,\ \{\Dx_i\hook F\ |\ i\}, \ \{\Dx_i\Dx_j\hook
            F\ |\ i, j\},\   \left\{ \Dx_i\Dx_j\Dx_k\hook F\ |\ i, j, k \right\}}
            =
            \\ &=
            \spann{F,\ \{\Dx_i\hook F\ |\ i\}} \oplus \Sym^{\leq 1} \VV =
            \spann{F_3 + F_2}\oplus \spann{\Dx_i\hook F_3\ |\ i} \oplus
            \Sym^{\leq 1} \VV.
        \end{align*}
        Therefore $\SS\hook F$ is uniquely defined by giving $F_3$ and $F_2
        \mod \spann{\Dx_i\hook F_3\ |\ i}$, up to multiplication by a constant.
    \end{remark}

    Finally we should give at least one explicit example.
    \begin{example}
        Let $F = x_1^3 +  \ldots  + x_n^3 \in \kk_{dp}[x_1, \ldots ,x_n]$. Then $\Ann{F} = (\Dx_i\Dx_j\ |\
        i\neq j) + (\Dx_i^3 - \Dx_j^3\ |\ i, j)$. Hence $H_{\Apolar{F}} = (1,
        n, n, 1)$.
    \end{example}

    \subsection{Relative Macaulay's inverse systems}\label{ssec:relative}
    \def\SymA{\Sym_A}%
    \def\divpow#1{P_{#1}}%
    \def\Homsheaf#1{\mathcal{H}om_{#1}}%
    \def\PPA{\divpow{A}}%
    \def\PPB{\divpow{B}}%
    \def\SSA{\SS_{A}}%


    The main result of this section is a local description of this universal
    family over the Gorenstein locus or, equivalently, any finite flat family
    with Gorenstein fibers (see Proposition~\ref{ref:localdescription:prop}
    and Proposition~\ref{ref:apolarflatness:prop}). This local description involves certain choices.
    This local description, in a special case of $A = \kk[[t]]$, first appeared
    in~\cite[Proposition~18]{emsalem}, unfortunately without a proof.

    With respect to the setting of the previous section there are two
    generalisations: first, we consider families instead of single subschemes
    and second, we consider all finite subschemes, in particular reducible
    ones (similar description of finite subschemes appeared independently
    in~\cite{Mourrain_polyexp}).
    We will see that both of there generalisations are not too hard, but in technical terms they
    force us to work with (divided powers) power series rings and base change, which
    accounts for some intransparency in the proofs.
    To help the reader separate the results of
    Proposition~\ref{ref:localdescription:prop} from the technicalities we provide
    the following intuitive comparison between families and apolar algebras.
    Here by dilation we mean the diagonal action of the torus $\kk^*$. All $Z$
    considered are finite. All families are over an affine, Noetherian base $A$.
    \begin{center}
    \vspace{-\baselineskip}
        \begin{tabular}[h]{@{} p{0.4\textwidth} lll p{0.4\textwidth}}
            Finite Gorenstein subschemes &&&& Polynomials/power series\\ \midrule
            $Z \subset \mathbb{C}^n$ && is equal to && $\Apolar{F}$
            for some
            $F\in \mathbb{C}_{dp}[[x_1, \ldots ,x_n]]$\\
            $Z \subset \mathbb{C}^n$ supported at the origin && is
            equal to && $\Apolar{F}$ for some
            $F\in \mathbb{C}_{dp}[x_1, \ldots ,x_n]$\\
            $Z \subset \mathbb{C}^n$ supported at the origin,
            invariant under dilation && is
            equal to && $\Apolar{F}$ for some homogeneous
            ${F\in \mathbb{C}_{dp}[x_1, \ldots ,x_n]}$\\
        \end{tabular}

        \medskip
        \begin{tabular}[h]{@{} p{0.4\textwidth} lll p{0.4\textwidth}}
            Families of Gorenstein subschemes &&&& Polynomials/power series\\ \midrule
            Family of $Z \subset \kk^n$ && is locally && $\Apolar{F}$
            for some
            $F\in A[[x_1, \ldots ,x_n]]$\\
            Family of $Z \subset \kk^n$, each $Z$ supported at the
            origin && is locally && $\Apolar{F}$
            for some
            $F\in A[x_1, \ldots ,x_n]$\\
            Family of $Z \subset \kk^n$, each $Z$ supported at the
            origin and invariant under dilation && is locally && $\Apolar{F}$
            for some homogeneous
            ${F\in A[x_1, \ldots ,x_n]}$\\
        \end{tabular}
    \end{center}
    Conversely, given $F\in A[[x_1, \ldots ,x_n]]$ as above one can wonder, whether $\Apolar{F}$ is a
    finite and flat family. This is a subtle question, but over a reduced ring
    $A$ it is necessary and enough that for each $t\in \Spec A$ the algebra $\Apolar{F_t}$ is finite of length independent of
    $t$. See Proposition~\ref{ref:apolarflatness:prop}.

    \subsubsection{Apolar families}
    Let $\kk$ be an arbitrary algebraically
    closed field. Let $A$ be a \emph{Noetherian}
    $\kk$-algebra and $\WW$ be a finite dimensional $\kk$-vector space.
    We define contraction action identically as in the previous section, but
    using coefficients from $A$ instead of $\kk$.
    Denote $\SymA(-) := {A\tensor_{\kk} \Sym(-)}$. Denote by $\SSA$ the space
    $\SymA \WW^*$ and by $\PPA$ the space $\Hom{A}(\SSA, A)$.
    We have
    a natural action of $\SSA$ on $\PPA$ by precomposition: for $s\in
    \SSA$ and $f:\SSA \to A$ we define $(sf)(t) = f(st)$.
    Abusing notation, we write $\Sym \WW$ for the divided powers ring. Then we
    may identify $\Sym \WW^*$ with $\left( \Sym \WW \right)^*$ and view
    elements of $\SymA \WW$ as elements of $\PPA$.
    In coordinates, the above description is much easier and down to earth:
    \[
        \SSA = A[\Dx_1, \ldots , \Dx_n]\qquad\mbox{and}\qquad \SymA \WW = A[x_1, \ldots ,x_n]
    \subset \PPA = A[[x_1, \ldots ,x_n]].\]
    The precomposition action is
    the $A$-linear action described by Formula~\eqref{eq:contraction}
    in the case $A = \kk$. Note that this action is $A$-linear, hence behaves
    well with respect to a base change.

    \def\ap#1{a_{#1}}%
    \begin{defn}
        For $F\in \PPA$ we define
        an ideal $\Ann{F}\subset \SSA$ by $\Ann{F} := \left\{ \sigma\in \SSA\ |\
        \sigma\hook F = 0 \right\}$ and call it the \emph{apolar ideal} of $F$.
        The algebra $\Apolar{F} := \SSA/\Ann{F}$ is called the \emph{apolar algebra} of $F$.
        It gives an affine \emph{apolar family}
        \[
            \ap{F}:\Spec \Apolar{F} \to \Spec A,
        \]
        which it is naturally embedded in $\WW \times \Spec A \to \Spec A$.
    \end{defn}

    Note that
    \begin{enumerate}
        \item if $F\in \SymA \WW$, then $\ap{F}$ is finite,
        \item $\SSA$-modules $\Apolar{F}$ and $\SSA\hook F$ are isomorphic.
    \end{enumerate}

    Different choices of $F$ lead to the same apolar family. But it is
    crucial for our purposes to keep track of $F$. So we always assume it to
    be fixed.
    \begin{convention}
        When speaking of $\Apolar{F}$ we always implicitly assume that $F$ is fixed.
    \end{convention}

    Let us see some examples.
    \begin{example}
        Assume $\operatorname{char} \kk \neq 2$.
        Take $A = \kk[t]$ and $F = tx_1^2 + x_1x_2\in A[x_1, x_2]$. Then
        $\SSA\hook F = \spann{F, x_1, x_2, 1}$ so that $\Ann{F} =
        (\Dx_1^2 - t\Dx_1\Dx_2,\,\Dx_2^2)$. The fiber of $\Spec\Apolar{F}$
        over $t = \lambda$ is $\kk[\Dx_1, \Dx_2]/(\Dx_1^2 -
        \lambda\Dx_1\Dx_2,\,\Dx_2^2) = \Apolar{\lambda x_1^2 + x_1x_2} = \Apolar{F_{\lambda}}$.
    \end{example}


    \begin{example}\label{ex:notionofflatness}
        Take $A = \kk[t]$ and $F = tx\in
        A[x]$. Then $\Ann{F} = (\Dx^2)$, so that the fibers of the associated
        apolar family are all equal to $\kk[\Dx]/\Dx^2$.
        For $t = \lambda$ non-zero we have $\kk[\Dx]/\Dx^2 = \Apolar{F_{\lambda}}$,
        but for $t = 0$ we have $F_{0} = 0$, so the fiber is \emph{not} the
        apolar algebra of $F_0$.
    \end{example}

    Intuitively, we think of $\Apolar{F}$ as a family of apolar algebras
    $\Apolar{F_t}$, where $t\in \Spec A$. But
    Example~\ref{ex:notionofflatness} shows that the fiber of $\Apolar{F}$
    over a point $t$ may differ from $\Apolar{F_t}$.
     To avoid such pathology we
    introduce an appropriate notion of continuity of apolar family
    $\Spec \Apolar{F}$. This notion, among other things, will guarantee that
    the fibers are apolar algebras of corresponding restrictions of $F$.

    \begin{defn}\label{ref:flatness:def}
        An apolar family $\Spec \Apolar{F}$ is \emph{flatly embedded} if the
        $A$-module $\PPA/(\SSA \hook F)$ is flat.
    \end{defn}
    Note that if an apolar family is flatly embedded, then $\SSA \hook F$ is
    a flat $A$-module, thus also the isomorphic module $\Apolar{F}$ is flat.
    Note also that in pathological Example~\ref{ex:notionofflatness} the
    $A$-module $\SSA\hook F = \Apolar{F}$ \emph{is} flat. Thus the flatness of
    $\SSA\hook F$ is strictly weaker than being flatly embedded.

    If $\Spec\Apolar{F}$ is flatly embedded, then its fiber over $t\in \Spec
    A$ is isomorphic to $\Spec \Apolar{F_t}$, see
    Proposition~\ref{ref:apolarflatness:prop} below.

    Definition~\ref{ref:flatness:def} is very similar to the definition of a
    subbundle, where we require that the cokernel is locally free. The
    technical difference is that the $A$-module $\PPA/(\SSA \hook F)$ is
    usually very far from being finitely generated.

    \subsubsection{Local description}
    We state the two main results of this section below. On one hand, we show
    that each finite flat family with Gorenstein fibers is locally a flatly
    embedded apolar family. On the other hand, we give descriptions of such
    families and their constructions over a reduced base.

    \begin{prop}[Local description of
            families]\label{ref:localdescription:prop}
            Let $T$ be a locally Noetherian scheme.
            Let ${X \subset \WW \times T\to T}$ be an embedded
        finite flat family of Gorenstein algebras.
        This family is isomorphic, after possibly
        shrinking $T$ to a neighborhood of a given closed point, to a flatly
        embedded finite family $\ap{F}:\Spec \Apolar{F} \to \Spec A$ for $F\in \PPA$.
        Moreover
        \begin{enumerate}
            \item if all fibres are irreducible supported at $0\in \WW$ then we may take $F\in \SymA \WW$.
            \item if all fibres are invariant under the
                dilation action of $\kk^*$ and $T$ is \emph{reduced}, then we may take a
                \emph{homogeneous} $F\in \SymA \WW$.
        \end{enumerate}
    \end{prop}

    \begin{example}
        Let $\WW = \mathbb{A}^1 = \Spec \kk[x]$. Consider a branched double cover
        \[X = \Spec \frac{\kk[\Dx, t]}{(\Dx^2 - t)} \to T = \Spec \kk[t].\]
        Then $X_0 = \Spec \kk[\Dx]/\Dx^2$ and $\omega_0$ is generated by $F_0 =
        x$. This $F_0$ lifts to an element $F = x\cdot \sum_{n\geq 0} (tx^2)^n$ of
        $\kk[t][[x]]$ and one checks that indeed
        $X = \Apolar{x\cdot \sum_{n\geq 0} (tx^2)^n}$. In this case no
        shrinking of $T$ is necessary.
    \end{example}

    Before we state the next result we explain one more piece of notation. Let
    $\varphi:A\to B$ be any homomorphism and $F\in \PPA$. By $\varphi(F)\in \PPB$ we
    denote the series obtained by applying $\varphi$ to coefficients of $F$.
    \begin{prop}\label{ref:apolarflatness:prop}
        Let $F\in \PPA$.
        Let $\ap{F}: \Spec \Apolar{F}\to \Spec A$ be an apolar family and
        suppose it is finite.
        Then the following holds:
        \begin{enumerate}
            \item If the apolar family is flatly embedded, then its base change via
                any homomorphism $\varphi:A\to B$ is equal to $\Spec \Apolar{\varphi(F)} \to
                \Spec B$. In particular, the fibre of $\ap{F}$ over $t\in
                \Spec A$ is naturally $\Spec \Apolar{F_t}$. Thus the
                length of $\Apolar{F_t}$ is independent of the choice of $t$.
            \item If $A$ is reduced and the length of $\Apolar{F_t}$ is
                independent of the choice of $t\in \Spec A$, then $\ap{F}$
                is flatly embedded.
        \end{enumerate}
    \end{prop}

    \begin{example}\label{ex:tangentvectors}
        \def\divexp#1{\operatorname{exp}_{dp}(#1)}%
        \def\ww{\mathbf{w}}%
        \def\xx{\mathbf{x}}%
        Let $\ww\in \WW$ and denote $\divexp{\ww} := \sum_{i\geq 0} \ww^i\in \PP$.

        Pick a polynomial $f\in \Sym \WW$. It gives a subscheme $\Spec
        \Apolar{f} \subset \WW$ supported at zero. We will now construct a
        family which moves the support of this subscheme along the line
        $\spann{\ww} \subset \WW$. The line is chosen for clarity only, the same
        procedure works for arbitrary subvariety of $\WW$.

        For every linear form $\Dx\in \WW^*$ we have
        \[\Dx\hook (f\divexp{\ww}) = (\Dx\hook f)\divexp{\ww} + c\cdot
        f\divexp{\ww}
        = ((\Dx + c)\hook f)\cdot \divexp{\ww},\] where $c = \Dx\hook \ww\in \kk$.
        Hence for every polynomial $\sigma(\xx)\in \Sym \WW^*$ we have
        $\sigma(\xx)\hook (f\divexp{\ww}) = 0$ if and only if $\sigma(\xx +
        \ww)\hook f = 0$.
        It follows that $\Spec \Apolar{f\divexp{\ww}}$ is the translate of
        $\Spec \Apolar{f}$ by the vector $\ww$, in particular it is supported on $\ww$ and abstractly
        isomorphic to $\Apolar{f}$. Hence $\Apolar{f}$ and
        $\Apolar{f\divexp{\ww}}$ have the same
        length.
        By Proposition~\ref{ref:apolarflatness:prop}, the family $\Spec \Apolar{f\divexp{t\ww}}\to \Spec \kk[t]$ is
        flatly embedded and geometrically corresponds to deformation by moving the support of
        $\Spec\Apolar{f}$ along the line spanned by $\ww$.
        Its restriction to $\kk[\varepsilon] = \kk[t]/t^2$ gives $\Spec \Apolar{f
        + \varepsilon \ww f}$ corresponding to the tangent vector pointing
        towards this deformation.
    \end{example}

    \subsubsection{Proofs of Proposition~\ref{ref:localdescription:prop} and
    Proposition~\ref{ref:apolarflatness:prop}}

    \begin{proof}[Proof of Proposition~\ref{ref:localdescription:prop}]
        Consider the $\OO_X$-sheaf $\omega = \Homsheaf{T}(\pi_*\OO_X,
        \OO_T)$, where the multiplication is by precomposition.
        Since the map $X\to T$ is finite and flat, the base change to a fiber over $t\in
        T$ is equal to $\Hom{\kappa(t)}(\OO_{X_t}, \kappa(t))$, where
        $\kappa(t)$ is the residue field in $t$. This is the
        canonical module of the fibre, see \cite[Chapter 21]{Eisenbud}.
        Note also that $\OO_X$ and $\omega$ taken as $\OO_T$-sheaves are
        locally free of the same rank.

        Choose a $\kk$-point $0\in T$. The $\OO_{X_0}$-module $\omega_{|X_0} =
        \Hom{\kappa(0)}(\OO_{X_0}, \kappa(0))$ is principal because the fibre $X_0$
        is Gorenstein.  Choose a generator $F$ of this module and lift it to
        $\omega$. By Nakayama's Lemma, after possibly shrinking $T$, we may assume
        that $F$ generates $\omega$ as an $\OO_X$-module.
        We get a surjection $\OO_X \onto \OO_X\cdot F = \omega$. This is in
        particular a surjection of locally free $\OO_T$-sheaves of the same
        rank. Hence it is an isomorphism, so that
        $\omega$ is a trivial $\OO_X$-invertible sheaf, trivialised by
        $F$.

        Shrinking $T$ even further, we may assume that $T = \Spec A$ is affine.
        Now, $X \subset \WW \times T$ is embedded as a $T$-subscheme.
        Therefore $\omega \subset \Homsheaf{T}(\OO_{\WW \times T}, \OO_T)$ as
        $\OO_{\WW \times T}$-modules.
        Since $\omega$ is annihilated by the ideal $I_{X} \subset
        \OO_{\WW \times T}$ and is torsion free as the $\OO_{X}$-module,
        we have $\Ann{\omega} = I_X$.
        Furthermore $F\in H^0(X, \omega) \subset \Hom{A}(\SSA, A) = \PPA$
        generates $\omega$ as a $\OO_X$-module or as an $\OO_{\WW\times
        T}$-module, hence $\Ann{\omega} = \Ann{F}$ and $X = \Spec \Apolar{F}$
        as families embedded into $\WW \times T$.


        Since $X\to T$ is flat, it follows that $\Apolar{F}$ is $A$-flat. The
        sequence
        \begin{equation}\label{eq:apolar}
            0\to \Ann{F} \to \SSA \to \Apolar{F} \to 0
        \end{equation}
        shows that
        $\Ann{F}$ is also $A$-flat. Moreover the $A$-module $\Apolar{F}$ is flat and finitely
        generated, hence is projective, so that $\SSA  \simeq \Ann{F} \oplus
        \Apolar{F}$ as $A$-modules.  Since $A$ is Noetherian the module
        $\Hom{A}(\SSA, A)$ is flat by~\cite[4.47,
        p.~139]{Lam__rings_and_modules}.
        Then $\Hom{A}(\Ann{F}, A)$ is flat as well, as a direct summand of a flat module.
        Applying $\Hom{A}(-, A)$ to the sequence~\eqref{eq:apolar}, which
        splits, we get
        \[
            0\to \Hom{A}(\Apolar{F}, A)\to \Hom{A}(\SSA, A) \to
            \Hom{A}(\Ann{F}, A) \to 0.
        \]
        By definition $\Hom{A}(\SSA, A) = \PPA$ and $\Hom{A}(\Apolar{F}, A) =
        H^0(\omega)$, which it is
        generated by $F$. Therefore the sequence can be written as
        \[
            0\to \SSA\hook F \to \PPA \to
            \Hom{A}(\Ann{F}, A) \to 0.
        \]
        Hence $\Hom{A}(\Ann{F}, A)  \simeq \PPA/(\SSA \hook F)$ so this last
        module is $A$-flat. By definition, $\Apolar{F}$ is flatly embedded.
        This concludes the general part of the proof.

        \def\mm{\mathfrak{m}}%
        \def\mmX{\mathfrak{m}_{|X}}%
        Now we pass to the proof of Points 1.~and 2.

        Point 1. Suppose, that all fibres of $X\to T$ are supported at the origin of
        $\WW$.  Take the ideal $\mm$ of the origin in $\WW \times T$ and its
        restriction to $X$, denoted $\mmX$.
        Each fiber $X_t$ over a closed
        point is a scheme of length $d$ supported at the origin, so $\mmX^d$ is zero
        on each fiber of $X\to T$. Thus $\mmX^d$ is contained in the radical.
        This radical is nilpotent, as $X$ is Noetherian being finite over
        Noetherian $T$. Therefore $\mmX^{de}$ is zero for $e$ large enough.
        In other words, $\mm^{de}$ annihilates $F$
        obtained above, so in particular all monomials of degree at least $de$
        annihilate it. Hence the corresponding coefficients of $F$ are zero
        and so $F\in \Sym_{A}^{<de} \WW \subset \SymA \WW$. Therefore $F$ is a
        polynomial.

        Point 2.
        By assumption $A$ is reduced.
        We already know that $\mm^d$ lies in
        $\Ann{F}$, where $d$ is large enough (in fact, by reducedness $d =
        \deg(X\to T)$ is sufficient), so that $F\in \SymA \WW$. Consider the
        image of the ideal $\Ann{F}$ in $\SSA/\mm^d$.  We want to show that it
        is homogeneous. We know that it is homogeneous when restricting to
        each fiber.
        Pick an element of this ideal and its homogeneous coordinate modulo
        the ideal. Such coordinate vanishes on each fiber, hence is nilpotent.
        As $A$ is reduced, it is zero. This shows that $\Ann{F}$ is
        homogeneous. Therefore $\SSA\hook F$ is spanned by homogeneous
        elements.  Now we may pick a homogeneous generator of the fiber and
        lift it (after possibly shrinking as above) to a homogeneous generator
        of $\SSA\hook F$.
    \end{proof}

    Proposition~\ref{ref:localdescription:prop} underlies our philosophy
    in the analysis of Hilbert scheme: we think of $\PP_A$ as a local
    version of a versal family (no uniqueness though).

    \begin{remark}
        The above proof shows that, without restricting $T$,
        each finite flat embedded and Gorenstein family $\pi:X\to T$ is isomorphic to $\Spec_T \Apolar{\omega}\to T$,
        where $\omega \subset \OO_T\tensor \Sym \VV$ is an invertible
        sheaf. We will not use this fact.
    \end{remark}

    The $A$-module $\PPA$ is far from being finitely generated and this technical problem needs to be
    dealt with before the next proof is conducted. In this following
    Lemma~\ref{ref:weierstrass:lem} we show that we can harmlessly replace $\PPA$ by a suitable finitely generated
    $A$-submodule of $\PPA$.
    \begin{lem}\label{ref:weierstrass:lem}
        \def\Ncoh{\PP_{A, N}}%
        Let $N \subset \PPA$ be a $\SSA$-module, which is finitely generated
        as an $A$-module. Then there exist a finitely generated flat $A$-submodule
        $\Ncoh \subset \PPA$ and a projection $\PPA\onto \Ncoh$, such that
        \begin{enumerate}
            \item\label{it:inj} The composition $N\into \PPA \onto \Ncoh$ is injective.
            \item\label{it:flatness} The module $\Ncoh/N$ is flat over $A$ if
                and only if the module $\PPA/N$ is flat over $A$.
            \item\label{it:dirsum} For every homomorphism $A\to B$ the induced
                map $B\tensor \Ncoh \to
                \PPB$ is injective and its image is a direct summand of
                $\PPB$ as a $B$-module.
        \end{enumerate}
    \end{lem}
    \begin{proof}
        \def\Ncoh{\PP_{A, N}}%
        \def\monos{\mathcal{M}}%
        \def\monosmod{A\langle \monos \rangle}%
        Since $N$ is finitely generated over $A$, also the $A$-module $\Hom{A}(N,
        N)$ is finitely generated.  The $\SSA$-module structure gives a map
        $\SSA \to \Hom{A}(N, N)$. Take a finite set of monomials in $\SSA$
        which generate the image as $A$-module and denote this set by $\monos$.
        For every monomial $m$ not in $\monos$ there are monomials $m_i\in
        \monos$ and coefficients $a_i\in A$ such that $m -
        \sum a_i m_i$ annihilates $N$.  Pick $f\in N$, then $(m - \sum a_i
        m_i)\hook f = 0$. But $(m - \sum a_im_i) \hook f$ is a functional,
        whose value on $1\in \SSA$ is $f(m - \sum a_im_i)$. Hence $f(m - \sum
        a_im_i) = 0$, so $f(m) = \sum a_i f(m_i)$. Thus the values of $f$ on monomials from $\monos$
        determine $f$ uniquely.
        In other words, the composition $N \into \PPA \to \Hom{A}(\monosmod, A)$
        is injective, where $\monosmod$ is the $A$-submodule of $\SSA$ with basis $\monos$.
        Let
        \[\Ncoh = \Hom{A}(\monosmod, A)\]
        with the projection $\PPA \to \Ncoh$ coming from restriction of
        functionals to $\monosmod$. Then Point~\ref{it:inj} follows by
        construction of $\Ncoh$.
        The restriction $\PPA\to \Hom{A}(\monosmod, A)$ has a natural section
        $\Hom{A}(\monosmod, A) \into \SymA \WW \subset \PPA$.
        We take this section to obtain an inclusion $\Ncoh \subset \PPA$.
        This section also proves that we have an inner direct product $\PPA
        \simeq \Ncoh \oplus K$ for an $A$-module $K$. Since $A$ is
        Noetherian the $A$-module $\PPA$ is \emph{flat} by~\cite[4.47,
        p.~139]{Lam__rings_and_modules}. Hence also $\Ncoh$ and $K$ are
        flat. Snake Lemma applied to the diagram
        \[
            \xymatrix{
                0 \ar[r] & N \ar[r]\ar[d]_{=} & \PPA \ar[d]\ar[r] & \PPA/N
                \ar[r]\ar[d] & 0\\
                0 \ar[r] & N \ar[r] & \Ncoh \ar[r] & \Ncoh/N \ar[r] & 0
            }
        \]
        gives an extension
        \[
            0\to K\to \PPA/N\to \Ncoh/N \to 0,
        \]
        whereas the projection $\PPA \to K$ trivialises it to a direct sum
        $\PPA/N \simeq
        \Ncoh/N \oplus K$.
        Since $K$ is flat over $A$ we get that $\Ncoh/N$ is flat over $A$ if and only if $\PPA/N$ is flat over
        $A$; this proves Point~\ref{it:flatness}.

        For the Point~\ref{it:dirsum} note that the homomorphism $B\tensor
        \Ncoh \to B\tensor \PPA\to \PPB$ sends $B\tensor \Hom{A}(\monosmod,
        A)$ isomorphically to the direct summand
        $\Hom{B}(B\langle\monos\rangle,
        B) \subset \PPB$, so this homomorphism is injective.
    \end{proof}

    \begin{proof}[Proof of Proposition~\ref{ref:apolarflatness:prop}]
        \def\PPAprim{\PP_{A, N}}%
        \def\SSB{\SS_B}%
        \def\len{\operatorname{len}}%

        The $A$-module $N = \SSA\hook F$ is finite by assumption. By
        Lemma~\ref{ref:weierstrass:lem} for $N = \SSA\hook F$ we get a module
        $\PP_{A, N}$.
        The Lemma says that:
        \begin{enumerate}
            \item $\PPAprim$ is flat and finite over $A$,
            \item The module $\PPA/N$ is flat if and only if the module
                $\PPAprim/N$ is flat,
            \item the homomorphism $B\tensor \PPAprim \to \PPB$ is injective.
        \end{enumerate}

        Proof of Point 1.
        Suppose that the apolar family is flatly embedded. Then $\PPAprim/N$
        is flat and $B\tensor N \to B\tensor \PPAprim$ is injective, so also
        $B\tensor N \to \PPB$ is injective. This morphism
        sends $F$ to $\varphi(F)\in \PPB$, hence sends $B\tensor N$ to $\SSB \hook \varphi(F)$.
        Thus $B\tensor N$ is isomorphic to $\SSB\hook \varphi(F)$, which in
        turn is isomorphic to $\Apolar{\varphi(F)}$.

        Proof of Point 2.
        As noted above, the family is flatly embedded if and only if
        $M := \PPAprim/N$ is flat.
        This module is finitely generated, hence it is
        flat if and only if it is locally free. Now $A$ is reduced, so this
        happens if and only if is has constant rank: the length of $M\tensor
        \kappa(t)$ is independent of the choice of $t\in \Spec A$. But this
        length is $\len \PPAprim - \len \SSB\hook \varphi(F) = \len \PPAprim -
        \len \Apolar{\varphi(F)}$. Hence
        this length is constant directly by assumption.
    \end{proof}




    \section{Analysis of the component $\Hilbshortother$}\label{ssec:proof}

    \def\Gtwosix{\operatorname{Gr}(2, \VV)}%
    \def\wedgetwo{\Lambda^2 \VV}%
    \def\wedgetwodual{\Lambda^2 \VV^*}%
    \def\cone#1{\operatorname{cone}(#1)}
    \def\GLV{\operatorname{GL}(\VV)}%
    \newcommand{\quotient}{\mathbin{\!/\mkern-4mu/\!}}%
    We now return to working over the field $\mathbb{C}$ with a fixed
    $6$-dimensional vector space $V$ and we use the notation
    stated in the introduction.

    We begin by defining the Iliev-Ranestad
    divisor $\DIR$. The Grassmannian $\Gtwosix\subset \mathbb{P}(\wedgetwo)$ is
    non-degenerate, arithmetically Gorenstein and of degree $14$. A
    $5$-dimensional projective subspace $\mathbb{P}W \subset
    \mathbb{P}(\wedgetwo)$ is \emph{admissible} if it does not intersect
    $\Gtwosix$. A general $\mathbb{P}W$ is admissible. For an admissible $\mathbb{P}W$
    the cone $Z = W \cap \cone{\Gtwosix}\subset \wedgetwo$ is a standard
    graded finite Gorenstein scheme $Z$ supported at the origin.
    For a general $\mathbb{P}^6$ containing a general admissible $\mathbb{P}W$, the
    intersection $\mathbb{P}^6 \cap \Gtwosix$ is a set of $14$ points and $Z$
    is a hyperplane section of the cone over these points, thus $Z$ is
    smoothable.
    Since $Z$ spans $\mathbb{A}^6$ and is of length $14$, one checks, using
    the symmetry of the Hilbert function, that $Z$ has Hilbert
    function $(1, 6, 6, 1)$. Therefore $Z =
    \Spec\Apolar{F_Z}$ for a cubic $F_Z\in \Sym^3 W^* \simeq \Sym^3 \mathbb{C}^6$,
    unique up to scalars and $Z$
    gives a well defined element $[F_Z]\in \cubicspace \quotient \GLV$.

    The set $\DIR \subset \cubicspace$
    is defined as the closure of the
    preimage of the set of such $[F_Z]$ obtained from all admissible
    $\mathbb{P}^5 = \mathbb{P}W$. It is called the \emph{Iliev-Ranestad}
    divisor, see~\cite{Ranestad_Voisin__VSP}.
    It is proven in~\cite[Lemma~1.4]{Ranestad_Iliev__VSPcubic} that $\DIR$ is a divisor
    and in \cite[Lemma~2.4]{Ranestad_Voisin__VSP} that $F$ lies in $\DIR$ if
    and only if it is apolar to a quartic surface scroll.

    \def\supp{\operatorname{supp}}%
    \def\Zzero{Z_0}%
    \def\Izero{I}%
    We will now rigorously prove several claims which together lead to the
    proof of Theorem~\ref{ref:mainthm:thm}. Our
    approach is partially based on the natural method of~\cite{CEVV}.
    Additional crucial steps are to prove that $\Hilbshort \setminus
    \Hilbshortzero$ is smooth and that $\Hilbshortzero \cap \Hilbshortother$
    is irreducible.

    In the first two steps we use the following abstract but trivial fact.
    \begin{lem}\label{ref:abstract:lem}
        Let $X$ and $Y$ be reduced, finite type schemes over $\mathbb{C}$. Let
        $X\to Y$, $Y\to X$ be two morphisms, which are bijective on closed
        points. If the composition $X\to Y\to X$ is equal to
        identity, then $X\to Y$ is an isomorphism.
    \end{lem}
    \begin{proof}
        Denote the morphisms by $i:X\to Y$ and $\pi:Y\to X$.
        The scheme-theoretical image of $i$ contains all closed points, hence
        is the whole $Y$. Therefore the pullback of functions via $i$ is
        injective. It is also surjective, since the pullback via composition
        $\pi\circ i$ is the identity. Hence $i$ is an isomorphism.
    \end{proof}
     In our setting, $X$ is a subset of the Hilbert scheme, $Y$ a
     subspace of polynomials and the maps are constructed using relative
     Macaulay's inverse systems.

    \paragraph{First, we identify $\exceptional$ with an open subset of
    $\cubicspace$.}
    We prove that associating to a general enough degree three polynomial its apolar
    algebra induces an isomorphism from $\cubicspacereg$ to
    $\exceptional$.

    \def\Funiv{\mathbf{F}}%
    \def\OO{\mathcal{O}}%
    \def\polyspacereg{\Sym^{\leq 3}_{max} \VV}%
    \def\cubiccone{\Sym^{3}_{max} \VV}%
    \begin{construction}\label{ref:dominatingcubics:constr}
        Let $i:\polyspacereg \into \Sym^{\leq 3} \VV$ denote the set of $F$ such that
        $\dim \left(\Sym V^*\hook F\right)$ is equal to $1 + 6 + 6 + 1 = 14$. Since $14$ is maximal
        possible, this set is open.
        Recall that by Remark~\ref{ref:associatedgraded:rem} a degree three polynomial $F$ lies in $\polyspacereg$ if and only if its
        leading form lies in $\polyspacereg$. Hence $\polyspacereg =
        \Sym^{\leq 2} \VV \times \cubiccone$.

        Denote for brevity $U = \polyspacereg$ and by $\Gamma(U)$ the global
        coordinate ring of this (affine) subvariety.
        Take an universal cubic $\Funiv\in \Gamma(U)\tensor \Sym^{\leq 3} \VV$,
        so that for every $u\in U$ we have $\Funiv(u) = i(u)$.
        Then by Proposition~\ref{ref:apolarflatness:prop} the family
        $\ap{\Funiv}:\Spec \Apolar{\Funiv} \to U = \polyspacereg$ is flatly
        embedded and gives a morphism
        \[
            \varphi:\polyspacereg \to \Hilbshort
        \]
        to the Gorenstein locus of the Hilbert scheme. The points
        of the image are just $\Spec \Apolar{\Funiv(u)}$, where $\Funiv(u)\in U$.  By
        Macaulay's inverse systems they correspond precisely to irreducible
        schemes with Hilbert function $(1, 6, 6, 1)$ and supported at the origin.

        Let $\mu:\Sym^{3} \VV \dashrightarrow \cubicspace$ be the
        projectivisation.
        Denote $\mu(U)$ by $\cubicspacereg$.
        The restriction of $\varphi$ to homogeneous cubics factors through $\mu$ and gives
        a morphism
        \[\varphi: \cubicspacereg \to \exceptional,\] which is bijective on
        points.
    \end{construction}
    \begin{claim}\label{claim:exceptional}
        The map $\varphi:\cubicspacereg \to \exceptional$ is an isomorphism.
    \end{claim}
    \begin{proof}
        Let $[Z]\in \exceptional$. Each fibre of the universal family
        over $\exceptional$ is $\mathbb{C}^*$-invariant thus the whole family is
        $\mathbb{C}^*$-invariant. By Local Description of Families
        (Proposition~\ref{ref:localdescription:prop}) near
        $[Z]$ this family has the form $\Spec \Apolar{F}
        \to \Spec A$ for some $F\in A\tensor \Sym^3 \VV$, so that
        $[F]$ gives a morphism $\Spec A \to
        \cubicspacereg$ which is locally an inverse to $\cubicspacereg \to
        \exceptional$. The claim follows from Lemma~\ref{ref:abstract:lem}.
    \end{proof}
    We will abuse notation and speak
    about elements of $\cubicspacereg$ being smoothable~etc.
    \def\Hilbotherred{\left( \Hilbshortother \right)_{red}}%
    \paragraph{Next, we construct the bundle $\Hilbotherred \to \exceptional$.}
    We now prove the following claim, which informally reduces the questions
    about $\Hilbshortother$ to the questions about $\exceptional$. Note that
    we will work on the reduced scheme $\Hilbotherred$, which eventually turns out
    to be equal to $\Hilbshortother$.
    \begin{claim}\label{claim:vectorbundle}
        $\Hilbotherred$ is a rank $21$ vector bundle over $\exceptional$ via a
        map $\pi:\Hilbotherred \to \exceptional$. This map on the level of
        points maps $[R]$ to $\Spec \gr H^0(R, \OO_R)$ supported at the origin
        of $\VV$. The schemes corresponding to points in the same fibre of
        $\pi$ are isomorphic.
    \end{claim}
    \newcommand{\tr}[1]{\operatorname{tr}(#1)}%
    \begin{proof}
        \def\UU{\mathcal{U}}%
        First we recall the support map, as defined in~\cite[5A]{CEVV}.
    Consider the universal family $\UU\to \Hilbshort$, which is flat.
    The multiplication by $\VV^*$ on $\OO_{\UU}$ is
    $\OO_{\Hilbshort}$-linear. The relative trace of such multiplication
    defines a map
    $\VV^*\to H^0(\Hilbshort, \OO_{\Hilbshort})$, thus a morphism
    $\Hilbshort\to \VV$. We restrict this morphism to $\Hilbshortother \to
    \VV$ and compose it with multiplication by
    $\frac{1}{14}$ on $\VV$ to obtain a map denoted $\supp$. If $[Z]\in
    \Hilbshortother$ corresponds to a scheme supported at
    $v\in \VV$, then for every $v^*\in \VV^*$ the multiplication by $v^* -
    v^*(v)$ is nilpotent on $Z$, hence traceless. Thus on $Z$, we have $\tr{v^*} = \tr{v^*(v)}
    = 14v^*(v)$ and $\supp([Z]) = v$ as expected.

    The support morphism $\supp:\Hilbshortother \to V$ is $(V, +)$
    equivariant, thus it is a trivial vector bundle:
    \[\Hilbshortother  \simeq V \times \supp^{-1}(0).\]
    \def\Hilbshortfibre{\Hilbshortother^0}%
    Restrict $\supp$ to $\Hilbotherred$ and consider the fibre
    $\Hilbshortfibre := \supp^{-1}(0)$. Since $\Hilbotherred$
    is reduced, also $\Hilbshortfibre$ is reduced. We
    will now use this in an essential way. By Local
    Description~(Proposition~\ref{ref:localdescription:prop}) the universal
    family over this scheme locally has the form $\ap{F}:\Spec \Apolar{F} \to
    \Spec A$ for some $F\in A\tensor \Sym \VV$. For every $p\in \Spec A$ we
    have $\deg F(p)\leq 3$. In other words, $F_{\geq 4}(p) = 0$ for all points
    in $A$. Since $A$ is reduced, we have $F_{\geq 4} = 0$, so that
    $\deg F\leq 3$. Let $F_3$ be the leading form. The fibres of
    $\gr{\ap{F}}:\Spec \Apolar{F_3}\to \Spec A$ and $\ap{F}$ are isomorphic. Since $A$ is
    reduced, by Proposition~\ref{ref:apolarflatness:prop} the family
    $\gr{\ap{F}}$ is also flat and gives a morphism
    $\Spec A \to \exceptional$. These morphisms glue to give a morphism
    \[
        \gr:\Hilbshortfibre \to \exceptional.
    \]
    \newcommand{\tdf}{\operatorname{tdf}}%
    Now we prove that $\gr$ makes $\Hilbshortfibre$ a vector bundle
    over $\exceptional$. Let $U = \polyspacereg$. By Construction~\ref{ref:dominatingcubics:constr} we have a
    $\varphi:U \to \Hilbshortfibre$ which is a surjection on points.
    This surjection comes from a flatly embedded apolar family $\Spec
    \Apolar{\Funiv}\to U$, where $\Funiv\in \Gamma(U)\tensor \polyspacereg$ is a
    universal cubic.
    \def\EEker{\mathcal{K}}%
    \def\EE{\mathcal{E}}%
    \def\PEE{\mathbb{P}\EE}%
    \def\SSplus{\SS^{+}}%

    Let $[Z]\in \Hilbshortfibre$ and $u\in U$ be a point
    in the preimage. Then $Z = \Apolar{\Funiv(u)}$ and $\gr([Z])$ is apolar to the
    cubic form $[\Funiv(u)_3] = [\Funiv_3(u)]$.
    Therefore $\gr \circ \varphi(\Funiv(u)) = [\Funiv_3(u)]$. In this way
    $U = \polyspacereg$ becomes a trivial vector bundle of rank $1 + 6 +
    \binom{7}{2} = 28$ over the cone $\cubiccone$ over $\cubicspacereg$.
    We will prove that $\Hilbshortfibre$ is a projectivisation of a quotient bundle of this bundle.
    Take a subbundle $\EEker$ of $U$ whose fiber over $F_3\in \cubiccone$
    is $(\Sym^{\geq 1} \VV^*)\hook F_3$.
    Then as in Remark~\ref{ref:uniquenessofcubics:rem} we see that the family $\Spec\Apolar\Funiv$ over $U$
    is pulled back from the quotient bundle $U/\EEker$, which we denote by
    $\EE$. Hence also the associated morphism $U\to \Hilbshortfibre$ factors
    as $U\to \EE \to \Hilbshortfibre$.
    Finally we may projectivise these bundles:
    we replace the
    polynomials in $\EE$ by their classes, obtaining a bundle over
    $\cubicspacereg$ which we denote, abusing notation, by $\PEE$. The
    morphism $\EE\to \Hilbshortfibre$
    factors as $\EE\to \PEE \to \Hilbshortfibre$. Finally we obtain
    \[
        \bar{\varphi}:\PEE\to \Hilbshortfibre.
    \]
    It is bijective on points (Remark~\ref{ref:uniquenessofcubics:rem}).

    By the Local Description of Families for every $[Z]\in \Hilbshortfibre$ we
    have a neighborhood $U$ so that the universal family is $\Spec
    \Apolar{F}\to U$ for $F\in H^0(U, \OO_U)\tensor \polyspacereg$. Then $F$ gives a map $U\to
    \polyspacereg$, thus $U\to \PEE$. This is a local inverse of
    $\bar{\varphi}$. Hence by Lemma~\ref{ref:abstract:lem} the variety $\Hilbshortfibre$ is isomorphic to the bundle
    $\PEE$ over $\exceptional$.
    To prove Claim~\ref{claim:vectorbundle} we define $\pi$ to be the
    composition of projection $V \times \Hilbshortfibre \to \Hilbshortfibre$
    and $\gr$. Since the former is a \emph{trivial} vector bundle and the
    latter is a vector bundle the composition is a vector bundle as well.

    Finally note that $\pi([Z])$ is isomorphic to the scheme $\Spec \gr
    H^0(Z, \OO_Z)$, which in turn is (abstractly) isomorphic to $Z$ by the discussion of
    Section~\ref{sec:prelims}. Thence all the schemes corresponding to points
    in the same fibre are isomorphic.
\end{proof}

\begin{cor}\label{ref:divisorinlocus:cor}
    The locus $\Hilbshortother \cap \Hilbshortzero \subset \Hilbshortother$
    contains a divisor, which is equal to $\pi^{-1}(\DIR)$, where $\DIR
    \subset \cubicspacereg$ is the restriction of the Iliev-Ranestad divisor.
\end{cor}
\begin{proof}
    By its construction, the divisor $\DIR \subset \cubicspacereg  \simeq \exceptional$
    parametrises smoothable schemes. By Claim~\ref{claim:vectorbundle} also
    schemes in $\pi^{-1}(\DIR)$ are smoothable, hence $\pi^{-1}(\DIR)$ is contained in
    $\Hilbshortother\cap \Hilbshortzero$. Again by Claim~\ref{claim:vectorbundle} this preimage is divisorial in
    $\Hilbshortother$.
\end{proof}

    \paragraph{Now we prove that $\Hilbshortother\setminus\Hilbshortzero$ is
    smooth, so $\Hilbshortother$ is reduced.}
    Let $Z \subset V$ be a finite irreducible (locally) Gorenstein scheme with
    Hilbert function $(1, 6, 6, 1)$. Let $\SS := H^0(V, \OO_V) = \Sym V^*$ and
    $R = H^0(Z, \OO_Z)$, then $R = \SS/I$. The tangent space to $\Hilbshort$
    at $[Z]$ is isomorphic to $\Hom{\SS/I}(I/I^2, \SS/I)$. Since $Z$ is
    Gorenstein, this space is dual to $I/I^2$. Note that
    $Z$ is isomorphic to $\Zzero = \Spec \gr R$ and $[\Zzero]\in \exceptional$.

    \begin{claim}\label{claim:singularlocus}
    $\Hilbshortother \cap \Sing \Hilbshort = \Hilbshortother \cap
    \Hilbshortzero = \pi^{-1}(\DIR)$ as sets. Therefore $\Hilbshortother$ is reduced.
    Moreover $\Hilbshortother \cap \Hilbshortzero \subset \Hilbshortother$ is
    a prime divisor.
    \end{claim}

    Being a singular point of $\Hilbshort$ and lying in
    $\Hilbshortzero$ are both independent of the embedding of a finite scheme.
    Hence all three sets appearing in the equality of
    Claim~\ref{claim:singularlocus} are preimages of their images in
    $\exceptional$. Therefore it is enough to prove
    the claim for elements of $\exceptional$.

    Take $[\Zzero]\in
    \exceptional$  with corresponding homogeneous ideal $\Izero$.
    Take $F\in \Sym^3 \VV$ so that $\Izero = \Ann{F}$.
    The point $[\Zzero]$ is smooth if and only if $\dim \SS/\Izero^2 = 76 + 14 = 90$.
    Consider the Hilbert series $H$ of $\SS/\Izero^2$.
    By degree reasons $\Izero^2$
    annihilates $\Sym^{\leq 3} \VV$. We now show that it annihilates also a
    $6$-dimension space of quartics. Notabene, by
    Example~\ref{ex:tangentvectors} this space is the tangent space to
    deformations of $\SS/\Izero$ obtained by moving its support in $\VV$.
    \begin{lem}\label{ref:quartics:lem}
        The ideal $\Izero^2$ annihilates the space $\VV\cdot F\subset \Sym^4
        \VV$.
    \end{lem}
    \begin{proof}
        Let $\Dx\in \VV^*$ and $x\in \VV$ be linear forms. Then $\Dx$ acts on $\Sym \VV$
        as a derivation, so that $\Dx\hook (xF) = (\Dx\hook x) F + x(\Dx\hook
        F) \equiv x(\Dx \hook F) \mod \SS \hook F$.
        Take any element $i\in \Izero_2$ and write it as $i = \sum
        \beta_i\beta_j$ with $\beta_i$ linear. Then
        \[
            i\hook (xF) = \sum \beta_i \hook (\beta_j\hook xF) \equiv \sum
            x(\beta_i\beta_j \hook F) = x(i\hook F) = 0\mod (\SS \hook F).
        \]
        Therefore $i\hook (xF) \in \SS\hook F$, hence is annihilated by
        $\Izero$. This proves that $\Izero_2\Izero$ annihilates $xF$. Other
        graded parts of $I^2$ annihilate $xF$ by degree reasons.
    \end{proof}

    By Lemma~\ref{ref:quartics:lem} and the discussion above we have
    $H_{Z_0} = (1, 6, 21, 56, r, *)$ with $r\geq 6$.
    Therefore $\sum H_{Z_0} = 84 + r$ and this equals $90$ if and only if $r = 6$ and $*$ consists of zeros.
    Now we show that if $r = 6$ then $*$ consists of zeros.
    See~\cite[Lemma~2.31]{emsalem_iarrobino_small_tangent_space} for related
    statement.
    \begin{lem}\label{ref:onlyquatricsmatter:lem}
        Let $F\in \Sym^3 \VV$ and $\Izero = \Ann{F}\subset \SS$ be as above. Suppose that
        $\dim\left(\Izero^2\right)^{\perp}_4 = 6$. Then ${\Sym^5 \VV^* \subset
        \Izero^2}$. In particular $H_{S/\Izero^2} = (1, 6, 21, 56, 6, 0)$, so
        that the tangent space to $\Hilbshort$ at $[\SS/\Izero]$ has dimension $76$.
        As a corollary, 
        $[\SS/\Izero]$ is singular if and only if
        $\dim\left(\Izero^2\right)^{\perp}_4 > 6$.
    \end{lem}
    \begin{proof}
        Suppose $\Sym^5 \VV^* \not\subset \Izero^2$ and take non-zero $G\in \Sym^5
        \VV^*$ annihilated by this ideal.
        By assumption $\dim\left(\Izero^2\right)^{\perp}_4 = 6$ and by
        Lemma~\ref{ref:quartics:lem} the $6$-dimensional space $\VV F$ is perpendicular to
        $\Izero^2$. Therefore $\left(\Izero^2\right)^{\perp} = \VV F$ and
        hence $\VV^*\hook G \subset \VV F$.

        We first show that all linear forms are partials of $G$, in other
        words that $\VV \subset \SS G$.
        Clearly ${0\neq \VV^* \hook G \subset \VV F}$.
        Take a non-zero $x\in \VV$ such that
        $xF$ is a partial of $G$.
        Let $W^* = \left(x^{\perp}\right)_1 \subset \VV^*$ be the
        space perpendicular to $x$.
        Let $xF = x^{e+1}\tilde{F}$, where $\tilde{F}$ is not divisible by
        $x$. Then there exists an element $\sigma$ of $\Sym W^*$ such that
        $\sigma\hook (xF) = x^{e+1}$. In particular $x\in \SS\hook (xF)$.
        Moreover $\SS_3 (xF) \equiv \SS_2 F = \VV \mod \mathbb{C}x$.
        Therefore, $\VV \subset \SS\hook G$, so $\VV = \SS_4\hook G$.
        By symmetry of the Hilbert function, $\dim \VV^* G = \dim \SS_4 G =
        6$. Since $\VV^* G$ is annihilated by $\Izero^2$, by comparing
        dimensions we conclude that
        \[\VV^* G = \VV F.\]

        Since $\Izero_3\hook(\Izero_2\hook G) = 0$ and $I_3^{\perp} =
        \mathbb{C}F$, we have $\Izero_2\hook G \subset
        \mathbb{C}F$, so $\dim \left(\SS_2\hook G + \mathbb{C}F\right) \leq 6 + 1= 7$.
        For every linear $\Dx\in \VV^*$ and $y\in \VV$ we have $\Dx\hook (yF)
        = (\Dx\hook y)F + y(\Dx\hook F) \equiv y(\Dx\hook F) \mod \mathbb{C}F$.
        Therefore we have
        \[
            \SS_2 \hook G = \VV^*\hook (\VV^*\hook G) = \VV^*\hook (\VV F)
            \equiv \VV(\VV^*\hook F)\mod
            \mathbb{C}F,
        \]
        thus $\dim \VV  (\VV^*\hook F) \leq 7$. Take any two quadrics $q_1,
        q_2\in \VV^* F$. Then $\VV q_1 \cap \VV q_2$ is non-zero, so that
        $q_1$ and $q_2$ have a common factor. We conclude that $\VV^* F =
        y\VV$ for some $y$, but then $\dim \VV(\VV^* F) = \dim y\Sym^2 \VV >
        7$, a contradiction.
    \end{proof}

    \def\Ffamily{\mathcal{F}}%
    \def\Itwobundle{\mathcal{I}_2}%
    \def\Singdiv{E}%
    As explained in Claim~\ref{claim:exceptional}, the map $[F]\to \Spec \Apolar{F}$ is an isomorphism
    $\cubicspacereg \to \exceptional$.
    Thus we may consider the statements $[\Zzero]\in \Hilbshortzero$ and
    $[\Zzero]$ is singular as conditions on the form $F\in \cubicspacereg$.
    For a family $\Ffamily$ of forms we get a rank $120$ bundle $\Sym^2\Itwobundle$ with
    an evaluation morphism
    \begin{equation}\label{eq:eval}
        ev:\Sym^2\Itwobundle \to (\VV \Ffamily)^{\perp} \subset \Sym^4 \VV^*.
    \end{equation}
    The
    condition $(\Izero^2)^{\perp}_4 > 6$ is equivalent to degeneration of $ev$
    on the fibre
    and thus it is \emph{divisorial} on $\cubicspacereg$. We now check
    that the associated divisor $\Singdiv = (\det ev = 0)= \Sing
    \Hilbshortother \cap \exceptional$ is prime of degree $10$.

    \begin{lem}\label{ref:intersectiondegree:lem}
        Fix a basis $x_0, \ldots ,x_5$ of $\VV$ and let $F = x_0x_1x_3 -
        x_0x_4^2 + x_1x_2^2 + x_2x_4x_5 + x_3x_5^2$. The line between $F$ and
        $x_5^3$ intersects $\Singdiv$ in a finite scheme of degree
        $10$ supported at $x_5^3$.
    \end{lem}
    \def\fixed{\mathrm{fixed}}%
    \begin{proof}
        \def\FF#1{F(#1)}
        First, one checks that every linear form is a partial of $F$, so that
        $F\in \cubicspacereg$.
        A minor technical remark: in this example we use contraction action. The same polynomials
        will work when considering partial differentiation action, but the
        apolar ideals will differ.
        Let
        \begin{align*}
            \fixed := \operatorname{span}(\alpha_0^2,\,
            \alpha_0\alpha_2,\, -\alpha_0\alpha_3+\alpha_2^2,\,
            &\alpha_0\alpha_4+\alpha_2\alpha_5,\,
            \alpha_0\alpha_5,\, \alpha_1^2,\, \alpha_1\alpha_2-\alpha_4\alpha_5,\,
            \alpha_1\alpha_3+\alpha_4^2,\\ &\alpha_1\alpha_4,\, \alpha_1\alpha_5,\,
            \alpha_2\alpha_3,\, \alpha_2\alpha_4-\alpha_3\alpha_5,\, \alpha_3^2,\,
        \alpha_3\alpha_4)
        \end{align*}
        be a $14$-dimensional space.
        Let $\FF{u, v} = uF + v x_5^3$. Then
        \[\fixed \oplus \mathbb{C}\left(v \alpha_3\alpha_5 + u \alpha_0\alpha_1 - u \alpha_5^2\right)
        \subset \Ann{F_2}\]
        and equality holds for a general choice of $(u:v)\in \mathbb{P}^1$.
        One verifies that the determinant of $ev$ restricted to this line is equal, up to unit, to $u^{10}$.
        This can be conveniently checked near $x_5^3$ by considering $\Sym^2 \Izero \to
        \Sym^4\VV^*/\VV x_5^3$ and near $F$ by $\Sym^2 \Izero \to
        \Sym^4\VV^*/\VV x_3x_5^2$. We note that the same equality holds in any
        characteristic other than $2,3$.
    \end{proof}

    \begin{remark}
        For $F$ as in Lemma~\ref{ref:intersectiondegree:lem} above $\Spec\Apolar{F}$ is an
        example of a non-smoothable scheme. Other examples may be easily
        obtained by choosing ``random'' degree three polynomials $F$.
    \end{remark}

    \begin{prop}\label{ref:divisorequality:prop}
        The divisor $\Singdiv = \Sing \Hilbshortother \cap \exceptional$ is
        prime of degree $10$.
        We have $\Singdiv = \Hilbshortzero\cap \exceptional = \DIR\cap
        \cubicspacereg$ as sets.
    \end{prop}

    \begin{proof}
        Take two forms $[F_1], [F_2]\in\cubicspace$ and consider the
        intersection of $E$ with the line $\ell$ spanned by them.
        By Lemma~\ref{ref:intersectiondegree:lem} the restriction to $\ell$ of the evaluation morphism
        from~Equation \eqref{eq:eval} is finite of degree $10$.
        Hence also $\Singdiv$ is of degree $10$.

        Note that $E$ is
        $\SL{V}$-invariant. By a direct check, e.g.~conducted
        with the help of computer (e.g.~LiE), we see that there are no $\SL{V}$-invariant
        polynomials in $\Sym^{\bullet}\Sym^3 \VV^*$ of degree less than ten.
        Therefore $E$ is prime.
        Now smoothable schemes are singular, hence we have
        ${\DIR}_{|\cubicspacereg} \subset
        \exceptional\cap \Hilbshortzero  \subset E$. Since
        ${\DIR}_{|\cubicspacereg}$ is also a
        non-zero divisor, we have equality of sets.
    \end{proof}

    \begin{remark}
        \def\II{\mathcal{I}}%
        In the proof of Proposition~\ref{ref:divisorequality:prop} (and so
        Lemma~\ref{ref:intersectiondegree:lem}) we could
        avoid calculating the precise degree of the restriction of $ev$ to
        $\ell$, provided that we prove that $(\det ev = 0)_{\ell}$ is zero-dimensional. Namely,
        let $\II \subset \Sym^{\bullet} V^{*}$
        be the relative apolar ideal sheaf on $\ell$. We look at $\II_2$. By
        the proof of Lemma we have $\II_2  \simeq  \OO^{14} \oplus \OO(-1)$.
        Hence $\Sym^2 \II_2$ has determinant $\OO(-16)$.
        Similarly $V\Ffamily  \simeq \OO(-1)^{6}$, hence
        $(V\Ffamily)^{\perp} = \OO^{114} \oplus \OO(-1)^{6}$ and thus
        $\det ev_{|\ell}:\OO(-16)\to \OO(-6)$ is zero
        on a degree $10$ divisor. We conclude that $\Singdiv$ has degree $10$.
    \end{remark}
    \begin{proof}[Proof of Claim~\ref{claim:singularlocus}]
        Schemes corresponding to different elements in the fiber of $\pi$ are
        abstractly isomorphic. Therefore,
        we have $\Sing \Hilbshortother = \pi^{-1}\pi(\Sing \Hilbshortother)$
        and $\Hilbshortzero\cap \Hilbshortother =
        \pi^{-1}\pi(\Hilbshortzero\cap \Hilbshortother)$. Hence the equality
        in Proposition~\ref{ref:divisorequality:prop} implies the equality in
        the Claim. The reducedness follows, because the scheme was defined via
        the closure: $\Hilbshortother = \overline{\Hilbshortother \setminus
        \Hilbshortzero}$. The last claim follows because $\DIR$ is prime and
        of codimension one thus its preimage under $\pi$ is also such.
    \end{proof}

    \begin{remark}
        If $[Z]\in \Hilbshortother$ lies in $\Hilbshortzero$, then the tangent
        space to $\Hilbshort$ at $[Z]$ has dimension at least $85 = 76 + 9$.
        This is explained geometrically by an elegant argument
        of \cite{Ranestad_Iliev__VSPcubic}, which we sketch below.
        Recall that we have an embedding $Z \subset  \mathbb{A}^6 \cap
        \wedgetwo$.
        Define a rational map $\varphi:\mathbb{P}(\wedgetwo) \dashrightarrow
        \mathbb{P}(\wedgetwodual)$ as the composition $\wedgetwo \to \Lambda^4 \VV  \simeq
        \Lambda^2 \VV^*$ where the first map is $w\to w\wedge w$ and the
        second is an isomorphism coming from a choice of element of $\Lambda^6
        \VV$. Then $\varphi$ is defined by the $15$ quadrics vanishing on
        $\Gtwosix$ and it is birational with a natural inverse given by
        quadrics vanishing on $Gr(2, V^*)$.
        The $15$ quadrics in the ideal of $Z$ are the restrictions of the $15$
        quadrics defining $\varphi$. Thus the map $\varphi:Z\to \wedgetwodual$
        is defined by $15$ quadrics in the ideal of $Z$.
        Since $\varphi^{-1}(\varphi(Z))$ spans at most an $\mathbb{A}^6$ the
        coordinates of $\varphi^{-1}$ give $15 - 6 = 9$ quadratic relations between
        those quadrics. Therefore $\dim I_4 \leq \dim \Sym^2 I_2 - 9 \leq 126
        - 15$ and $r$ from the discussion above
        Lemma~\ref{ref:onlyquatricsmatter:lem} is at least $15$ so that the tangent space
        dimension is at least $85$.
    \end{remark}

    We conclude by making a formal proof of the main theorem and adding some remarks.
    \begin{proof}[Proof of Theorem~\ref{ref:mainthm:thm}]
        2.~By Claim~\ref{claim:singularlocus} the component $\Hilbshortother$ is
        reduced. Hence this part follows by Claim~\ref{claim:vectorbundle}.

        3.~This is proved in Claim~\ref{claim:exceptional}.

        1. $\Hilbshortother$ is smooth and connected, being a vector bundle
        over $\cubicspacereg$.

        4. This is proved in Claim~\ref{claim:singularlocus}.
    \end{proof}

        \begin{remark}\label{ref:chilozuccero:rmk}
            In \cite{CHILO__symmetrizing_matrix_multiplication} Chiantini,
            Hauenstein, Ikenmeyer, Landsberg and Ottaviani give very strong numerical evidence that
            $\secantvariety$ is an ideal-theoretic complete intersection of
            a degree $10$ and degree $28$ divisors. Since the degree $10$
            divisor is $\SL{V}$-invariant, it is equal to $\DIR$. It would be
            interesting to know whether the second divisor is $\DVap$
            introduced in~\cite{Ranestad_Voisin__VSP}.

            The outline of the argument of \cite{CHILO__symmetrizing_matrix_multiplication},
            as communicated by the authors, is to use \cite{hauenstein_sommese__witness} to show that
            $\secantvariety$ has codimension $2$ and degree $280$. Then \cite{daleo__hauenstein__aCM, daleo__hauenstein__aG} are
            used to show that it is
            arithmetically Gorenstein. Finally, the Hilbert function of the
            $280$ witness points allows one to conclude that it is actually a
            complete intersection of a
            degree $10$ and a degree $28$ polynomial.
        \end{remark}

    \begin{remark}\label{ref:uniquesmall:rmk}
        \def\ZZ{\mathcal{Z}}%
        The dimension of $\Hilbshortother$ equal to $76$ is smaller than the
        dimension of $\Hilbshortzero$ equal to $14\cdot 6 = 84$. This is the
        only known example of a component $\ZZ$ of the Gorenstein locus of
        Hilbert scheme of $d$ points on $\mathbb{A}^n$ such
        that $\dim \ZZ \leq dn$ and points of $\ZZ$ correspond to \emph{irreducible}
        subschemes. It is an interesting question whether other examples
        exist.
    \end{remark}

\small
\newcommand{\etalchar}[1]{$^{#1}$}
\def\cdprime{$''$}

\end{document}